\documentclass[11pt, a4paper,reqno]{amsart}
\usepackage{amsmath}
\usepackage{amsfonts}
\usepackage{amssymb}
\usepackage{amsthm}
\usepackage[active]{srcltx}

\numberwithin{equation}{section}

\newtheorem{theorem}{Theorem}[section]
\newtheorem{lemma}[theorem]{Lemma}

\newtheorem{corollary}[theorem]{Corollary}

\theoremstyle{definition}
\newtheorem{definition}[theorem]{Definition}

\theoremstyle{remark}
\newtheorem{remark}[theorem]{Remark}

\def\kint_#1{\mathchoice%
         {\mathop{\kern 0.2em\vrule width 0.6em height 0.69678ex depth -0.58065ex
                 \kern -0.8em \intop}\nolimits_{\kern -0.4em#1}}%
         {\mathop{\kern 0.1em\vrule width 0.5em height 0.69678ex depth -0.60387ex
                 \kern -0.6em \intop}\nolimits_{#1}}%
         {\mathop{\kern 0.1em\vrule width 0.5em height 0.69678ex depth -0.60387ex
                 \kern -0.6em \intop}\nolimits_{#1}}%
         {\mathop{\kern 0.1em\vrule width 0.5em height 0.69678ex depth -0.60387ex
                 \kern -0.6em \intop}\nolimits_{#1}}}
\def\vintslides_#1{\mathchoice%
         {\mathop{\kern 0.1em\vrule width 0.5em height 0.697ex depth -0.581ex
                 \kern -0.6em \intop}\nolimits_{\kern -0.4em#1}}%
         {\mathop{\kern 0.1em\vrule width 0.3em height 0.697ex depth -0.604ex
                 \kern -0.4em \intop}\nolimits_{#1}}%
         {\mathop{\kern 0.1em\vrule width 0.3em height 0.697ex depth -0.604ex
                 \kern -0.4em \intop}\nolimits_{#1}}%
         {\mathop{\kern 0.1em\vrule width 0.3em height 0.697ex depth -0.604ex
                 \kern -0.4em \intop}\nolimits_{#1}}}

\newcommand{\vp}{\varphi}

\newcommand{\be}{\begin{equation}}
\newcommand{\ee}{\end{equation}}

\newcommand{\bes}{\begin{equation*}}
\newcommand{\ees}{\end{equation*}}

\newcommand{\R}{\mathbb{R}}
\newcommand{\Rn}{\mathbb{R}^d}

\newcommand{\norm}[1]{\left| #1 \right|}

\newcommand{\abs}[1]{\left| #1 \right|}

\newcommand{\esssup}{\operatornamewithlimits{ess\, sup}}
\newcommand{\essinf}{\operatornamewithlimits{ess\,inf}}
\newcommand{\essosc}{\operatornamewithlimits{ess\, osc}}

\newcommand{\divt}{\operatorname{div}}
\renewcommand{\div}{\nabla \cdot}

\renewcommand{\l}{\left}
\renewcommand{\r}{\right}

\newcommand{\parts}[2]{\frac{\partial {#1}}{\partial {#2}}}
\newcommand{\pdo}[2]{\frac{\partial #1}{\partial #2}}

\providecommand{\brc}[1]{\left\lbrace#1\right\rbrace}

\def\Xint#1{\mathchoice
{\XXint\displaystyle\textstyle{#1}}%
{\XXint\textstyle\scriptstyle{#1}}%
{\XXint\scriptstyle\scriptscriptstyle{#1}}%
{\XXint\scriptscriptstyle\scriptscriptstyle{#1}}%
\!\int}

\def\XXint#1#2#3{{\setbox0=\hbox{$#1{#2#3}{\int}$}
\vcenter{\hbox{$#2#3$}}\kern-.5\wd0}}

\def\dashint{\Xint-}

\title[Local H\"older continuity for DNPE]{Local H\"older continuity for some doubly nonlinear parabolic equations in measure spaces}
\author[Henriques and Laleoglu]{Eurica Henriques and Rojbin Laleoglu}

\begin{document}

\begin{abstract}

We establish the local H\"older continuity for the nonnegative weak solutions of certain doubly nonlinear parabolic equations possessing a singularity in the time derivative part and a degeneracy in the principal part. The proof involves the method of intrinsic scaling and the setting is a measure space equipped with a doubling non-trivial Borel measure supporting a Poincar\'e inequality.

\end{abstract}

\date{\today}

\keywords{H\"older continuity, singular PDE, degenerate PDE, intrinsic scaling}

\subjclass[2000]{Primary 35B65. Secondary 35K65, 35K67, 35D10}

\thanks{This research was funded by the Portuguese Government through the FCT, Funda\c c\~ao para a Ci\^encia e a Tecnologia, under the project PEst-OE/MAT/UI4080/2011 and projects UTAustin/MAT/0035/2008 and PTDC/MAT/098060/2008.}

\maketitle

\section{Introduction}

We consider the regularity question for the nonnegative weak solutions of the doubly nonlinear equation
\begin{equation}\label{equation}
\frac{\partial (u^{q})}{\partial t}-\div{(|\nabla u|^{p-2}\nabla u)}=0, \quad p>2 \quad \text{and} \quad 0<q<1
\end{equation}
which appears, together with some other similar ones, in the modeling of turbulent filtration of non-Newtonian fluids through a porous media (see~\cite{DiazThelin94}). We show that the nonnegative weak solutions are locally H\"older continuous in measure spaces assuming only the measure to be a doubling non-trivial Borel measure supporting a Poincar\'e inequality.

In the simpler case $q=1$, the equation becomes the heat equation for $p=2$ and the evolutionary $p$-Laplace equation for $1<p\neq2$. For the indicated ranges of $p$ and $q$, in addition to the double nonlinearity, this  equation possesses a degeneracy, coming from the fact that its modulus of ellipticity $|\nabla u|^{p-2}$ vanishes at points where $|\nabla u|=0$, and a singularity due to the term $u^{q-1}$, enclosed in the time derivative part, which blows up at points where $u=0$. The presence of the degeneracy and the singularity in~\eqref{equation} together with the available literature (see~\cite{DiBe93, DiBeUrbaVesp04, Urba08}), suggest the use of the method of intrinsic scaling to show that weak solutions are locally continuous. However, in order to deal simultaneously with the double degradation of the parabolic structure of \eqref{equation} one needs to consider several new aspects in the argument, (for other successful uses of intrinsic geometry, see also \cite{AcerMing07,BogeDuzaMing11,DiBeGianVesp08,DiBeGianVesp10a,HenrUrba05,Kuus08}). 

A better understanding of the adopted procedure requires the comprehension of the dichotomy in the behaviour of weak solutions to equation \eqref{equation}: for large scales, when the infimum is considerably smaller than the supremum, the singularity in the time derivative part dominates; while for small scales the equation becomes essentially the degenerate $p$-Laplace equation. Although both cases are treated via the method of intrinsic scaling, each one requires the construction of a suitable intrinsic geometry on its own.

A similar type of behaviour is displayed by the so-called Trudinger's equation which agrees with \eqref{equation} for $q=p-1$. The local H\"{o}lder continuity for the weak solutions of so-called degenerate Trudinger's equation was established in \cite{KuusSiljUrba10} while the complementary case that corresponds to the singular one was proven in \cite{KLSU}. Like~\eqref{equation}, Trudinger's equation becomes essentially the degenerate (singular) $p$-Laplace equation when the solution is strictly away from zero whereas the degeneracy (singularity) dominates otherwise for $p>2$ ($1<p<2$). Nevertheless, unlike~\eqref{equation}, Trudinger's equation is homogeneous of degree one and this favored feature brings some convenience in the analysis since it creates some resemblance with the heat equation. For example, similar to the heat equation, it admits a scale and location invariant parabolic Harnack inequality (see~\cite{Trud68,KinnKuus07}) which plays a crucial role in the proof of H\"older continuity in the most involved case, namely in the case when the infimum is considerably smaller than the supremum of the solution. For equation~\eqref{equation} such a Harnack inequality is not available (see~\cite{FornGian10,FornSosi08}) therefore, a new approach is required to deal with the singularity in the time derivative part. Consequently, despite the similarities in the analysis, the techniques employed are respectably different. We get through the hitch by making use of logarithmic estimates in a modified setting and defining the correct intrinsic geometry imposed by the energy estimates that are highly inhomogeneous.

The problems concerning the continuity of the solutions of partial differential equations presenting a double nonlinearity were also addressed by Porzio and Vespri, Vespri and Ivanov, see~\cite{PorzVesp93, Vesp92, Ivan94}, for equations that are equivalent to~\eqref{equation} provided the solution is strictly away from zero, which is not our case.

The paper is organized as follows. In section 2 we provide the basic definitions and facts on metric spaces which will be needed along the proof, and we also present the main result. In section 3 we construct the iteration argument and the fundamental estimates, namely the energy and logarithmic estimates; section 4 includes the implementation of the alternative argument that will lead to the proof of the continuity of the nonnegative local weak solutions of ~\eqref{equation}.

Pretending this work to be the most self-contained possible, there will be natural repetitions of statements and results already published (which will come together with their respective references). Meanwhile, we tried to keep the paper mostly original and confined ourselves with just referring to the corresponding references for some arguments and highlighted the novelties of this work. 

\section{Preliminary material and main result} \label{section:preliminaries}

Let $\mu$ be a Borel measure and $U$ be an open set in $\Rn$.
The Sobolev space $H^{1,p}(U;\mu)$ is defined to be the completion of $C^\infty(U)$ with respect to the Sobolev norm
\[
\|u\|_{1,p,U}=\left(\int_U(|u|^p + |\nabla u|^p)\, d\mu\right)^{1/p}.
\]
A function $u$ belongs to the local Sobolev space $H_{loc}^{1,p}(U;\mu)$ if it belongs to
$H^{1,p}(\Omega;\mu)$ for every compactly contained subset $\Omega$ of $U$. Moreover, the Sobolev space with zero boundary values $H^{1,p}_0(U;\mu)$ is defined as the completion of $C_0^\infty(U)$ with respect to the Sobolev norm (see e.g. \cite{HeinKilpMart93} for more properties of Sobolev spaces).

The parabolic Sobolev space $L^p(t_1,t_2;H^{1,p}(U;\mu))$, with $t_1<t_2$, is the space of functions $u(x,t)$
such that, for almost every $t\in(t_1,t_2)$ the function $u(\cdot, t)$ belongs to $H^{1,p}(U;\mu)$ and
\[
\int_{t_1}^{t_2}\int_U(|u|^p +|\nabla u|^p) \, d\nu < \infty,
\]
where we denote $d\nu=d\mu\,dt$.

The space $L_{loc}^p(t_1,t_2;H_{loc}^{1,p}(U;\mu))$ is defined analogously.

\begin{definition}
A function $u \in L_{loc}^p(\tau_1,\tau_2;H_{loc}^{1,p}(U;\mu))$, where $U$ is an open and bounded subset of $\Rn$, is a
weak solution of equation \eqref{equation} in $U\times(\tau_1,\tau_2)$ if for any compact $\Omega \subset U$ and for almost every $\tau_1 <t_1 < t_2 < \tau_2$ it satisfies the integral identity
\begin{equation} \label{weak_solution}
\begin{split}
& \qquad \int_{\Omega} \left[u^q(x,t_2)\eta(x,t_2) - u^q(x,t_1)\eta(x,t_1)\right]\,d\mu  \\
&+\int_{t_1}^{t_2}\int_{\Omega} \left( |\nabla u|^{p-2}\nabla u\cdot \nabla \eta
-u^{q} \frac{\partial\eta}{\partial t} \right)\, d\nu = 0
\end{split}
\end{equation}
for every  $\eta \in C_0^\infty(\Omega \times (\tau_1,\tau_2))$.
\end{definition}

Here the boundary terms are taken in the sense of limits
$$
\int_{\Omega} u^q(x,t_1)\eta(x,t_1) \,d\mu = \lim_{h\rightarrow 0} \int_{t_1}^{t_1+h}\int_{\Omega} u^q(x,t)\eta(x,t) \,d\nu
$$
and
$$
\int_{\Omega} u^q(x,t_2)\eta(x,t_2) \,d\mu = \lim_{h\rightarrow 0} \int_{t_2-h}^{t_2}\int_{\Omega} u^q(x,t)\eta(x,t) \,d\nu.
$$
The following notion of weak solution given by the Steklov averages is technically more convenient to work with since it involves the discrete time derivative of the solution and it is equivalent to the previous definition~\cite[p. 101--102]{DiBe93}.

\begin{definition}\label{d:ws}
A weak solution of~\eqref{equation} is a measurable function $u \in L_{loc}^p(\tau_1,\tau_2;H_{loc}^{1,p}(U;\mu))$ satisfying
\begin{equation}\label{ws_a}
\int_{\Omega\times\left\lbrace t\right\rbrace} \! \left\lbrace \pdo{\l(u^q\r)_{h}}{t}\eta\,+\,
\left(\abs{\nabla u}^{p-2}\nabla u\right)_{h}\cdot \nabla \eta \right\rbrace \,d\mu=0,
\end{equation}
for every compact set $\Omega\subset U$, for almost every $\tau_1<t<\tau_2-h$ and for all $\eta \in C^{\infty}_{0}(\Omega).$
\end{definition}

The following definitions and results can be found within the context of analysis on metric measure spaces.

The measure $\mu$ is \emph{doubling} if there is a universal constant $D_0\ge 1$ such that
\[
\mu(B(x,2r))\le D_0 \mu(B(x,r)),
\]
for every $B(x,2r)\subset\Omega$, denoting
\[
B(x,r):= \left\{ y\in \mathbb{R}^d \, :\, \norm{y-x}<r\right\}
\]
the standard open ball in $\Rn$ with radius $r$ and center $x$.

The dimension related to the doubling measure is $d_\mu=\log_2 D_0$; in the particular case of the Lebesgue measure, $\mathcal{L}$, one has $d_\mathcal{L}=d$.

Let $0<r<R<\infty$.
A simple iteration of the doubling condition implies that
\begin{equation}\label{iteration_dp}
\frac{\mu(B(x,R))}{\mu(B(x,r))}
\le C\left(\frac Rr\right)^{d_\mu},
\end{equation}
where $C$ depends only on the doubling constant $D_0$.

For $\alpha\in (0,1]$, we say that the measure $\mu$ satisfies the \emph{$\alpha$-annular decay property} if there exists a constant $c \geq 1$ such that
\begin{equation}\label{annular_decay}
\mu(B(x,r)\setminus B(x,(1-\delta)r))\le c\delta^\alpha\mu(B(x,r)),
\end{equation}
for all $B(x,r)\subset\Omega$ and for all $\delta\in(0,1)$. Obviously, in our setting this property holds since it is an immediate consequence of the doubling condition, in particular of~\eqref{iteration_dp}. Recall also that it was proven by Buckley \cite{Buck99} that any length space, a metric space in which the distance between any pair of points is the infimum of the lengths of rectifiable paths joining them, with a doubling measure supports such an $\alpha$-annular decay property for some $\alpha\in (0, 1]$.

The measure is said to support \emph{a weak $(q,p)$-Poincar\'e inequality} if there exist constants $P_0>0$ and $\tau\ge 1$ such that
\begin{equation}\label{poincare}
\l(\dashint_{B(x,r)}|u-u_{B(x,r)}|^q \, d\mu  \r)^{1/q}
\le P_0 r\left(\dashint_{B(x,\tau r)} |\nabla u|^p \, d\mu\right)^{1/p},
\end{equation}
for every $u\in H^{1,p}_{loc}(\Omega;\mu)$ and $B(x,\tau r)\subset\Omega$, where
\[
u_{B(x,r)}=\dashint_{B(x,r)} u \, d\mu = \frac{1}{\mu(B(x,r))}\int_{B(x,r)} u\, d\mu.
\]
Since the constant $\tau$ may be strictly greater than one, the inequality gains the extra word "weak". It is known (see Theorem 3.4 in \cite{HajlKosk00}) that in $\Rn$ with a doubling measure, the weak $(q,p)$-Poincar\'e inequality with some $\tau\ge1$ implies the $(q,p)$-Poincar\'e inequality with $\tau=1$, therefore we can consider $\tau=1$.

We will assume that the measure supports a weak $(1,p)$-Poincar\'e inequality. The weak $(1,p)$-Poincar\'e inequality together with the doubling condition
imply a weak $(\kappa,p)$-Sobolev-Poincar\'e inequality with
\begin{equation} \label{kappa}
\kappa =
\begin{cases}
\dfrac{d_\mu p}{d_\mu -p}, & 1<p< d_\mu, \\
2p, & p \ge d_\mu,
\end{cases}
\end{equation}
where $d_\mu$ is as above (the proof can be found in \cite{HajlKosk00}).

For Sobolev functions with zero boundary values,
we have the following version of Sobolev's inequality (see \cite{KinnShan01}):
suppose that $u\in  H_0^{1,p}(B(x,r);\mu)$, then
\begin{equation}\label{Sobolevzero}
\left(\dashint_{B(x,r)} |u|^\kappa\, d \mu
\right)^{1/\kappa} \le C r \left(\dashint_{B(x,r)}
|\nabla u|^p\, d \mu \right)^{1/p}.
\end{equation}
Moreover, the weak $(1,p)$-Poincar\'e inequality and the doubling condition also imply
the $(1,q)$-Poincar\'e inequality for some $q<p$~\cite{KeitZhon08}. Consequently, we also have the weak $(\kappa,q)$-Sobolev-Poincar\'e inequality which implies, by means of H\"older's inequality, the $(q,q)$-Poincar\'e inequality for some $q<p$. For more information on Sobolev-type inequalities see \cite{SaCo92,SaCo02}.

We would like to highlight the fact that, even though we worked in subdomains of $\Rn$ we constructed all our arguments so that they can be adjusted to any other metric measure space with a doubling measure supporting a Poincar\'e inequality and an $\alpha$-annular decay property for some $\alpha\in (0, 1]$. As mentioned before, the last property holds true for length spaces equipped with a doubling measure. Moreover, the fact that the theory of Sobolev spaces is based on the notion of weak derivatives brings about the need of an alternative way of defining weak gradients in a general metric measure space because the concept of direction is not always clear in such a context. However, this can be overcome by employing the so-called upper gradients. There are several recent contributions to the development of the theory of Sobolev--type metric spaces. For further knowledge on the subject we refer to the survey \cite{Hajl03} and  for examples of metric spaces equipped with a doubling measure and satisfying a weak Poincar\'e inequality we refer to~\cite{KinnShan01}.

Our main result is the following theorem.
\begin{theorem}\label{main_theorem}
Let $p > 2$ and $0<q<1$. Assume that the measure $\mu$ is doubling and supports a weak $(1,p)$-Poincar\'e inequality. Then any non-negative bounded weak solution of equation~\eqref{equation} is locally H\"older continuous.
\end{theorem}

\begin{remark}
Theorem \ref{main_theorem} still holds for equations with a more general principal part $ \mathcal{A}$
\[
\frac{\partial (v^{q})}{\partial t}-\divt \mathcal{A}(x,t,v,\nabla v) =0,
\]
where the Carath\'eodory function $\mathcal{A}$ satisfies the structure conditions
\begin{align}
&\mathcal{A}(x,t,v,\eta)\cdot \eta \geq \mathcal{A}_0 \norm{\eta}^p, \label{structure_1}\\
&\mathcal{A}(x,t,v,\eta)\leq \mathcal{A}_1 \norm{\eta}^{p-1}, \label{structure_2}
\end{align}
for almost every $(x,t) \in \R^{n} \times \R$ and every $(v,\eta) \in \R \times \R^n$, for
some constants $\mathcal{A}_0,\mathcal{A}_1>0$.
\end{remark}

Along the text we will work with parabolic cylinders built upon balls constructed as follows. Let $(x_{0},t_{0})$ be a point in the space-time domain. The cylinder of radius $r>0$ and height $s>0$, with vertex at $(x_{0},t_{0})$, is defined as
\begin{align*}
Q_{x_{0},t_{0}}(s,r)&:= B(x_0,r) \times (t_{0}-s,t_{0}) .
\end{align*}
We write $Q\left(s,r\right)$ to denote $Q_{0,0}(s,r)$.
Moreover, we shall use the notation
$\delta Q_{x_{0},t_{0}}(s,r) = Q_{x_{0},t_{0}}(\delta^p s,\delta r)$ and $\delta B(x_0,r) = B(x_0,\delta r)$.

In the sequel, we call \textit{data} the set of a priori constants
$p$, $q$, $d$, $D_0$, and $P_0$; we will assume the weak solutions of equation~\eqref{equation} to be locally bounded taking \begin{equation}\label{sup_u}
\|u\|_{{L}^{\infty}(U\times (\tau_1,\tau_2))}\leq M, \quad \mbox{for some} \  M>0.
\end{equation}
This assumption will be used without further comments hereafter.

To close this section of preliminary material we need to present the following algebraic lemma, the so-called ``lemma on fast geometric convergence'' (cf. \cite{DiBe93}).

\begin{lemma}\label{geometric_convergence}
Let $(Y_n)_n$ be a sequence of positive numbers satisfying
$$
Y_{n+1}\le Cb^nY_n^{1+\alpha},
$$
where $C,b>1$ and $\alpha>0$. Then $(Y_n)_n$ converges to zero as $n\rightarrow\infty$ provided
$$
Y_0\le C^{-1/\alpha}b^{1-\alpha^2}.
$$
\end{lemma}


\section{Accommodating the double nonlinearity and presenting the logarithmic and energy estimates}

Let $u$ be a nonnegative weak solution of equation \eqref{equation} in $U\times(\tau_1,\tau_2)$ and let $K$ be a compact subset of $U\times (\tau_1,\tau_2)$. As it is now a standard procedure, the H\"older continuity of $u$ at a point $(x_0,t_0)$ in $K$ follows via an iteration process applied in a sequence of nested and shrinking cylinders with vertex at that point. In each step of the iteration process, we prove that the oscillation of the solution reduces in a measurable way as we suitably decrease the size of the cylinder. Ultimately, this will lead to the conclusion that the oscillation converges to zero as the cylinders shrink to the point. Because it is always possible to translate the equation, without loss of generality, we restrict the study to the origin. Moreover, since we are focused on the local interior regularity, it is enough for our purposes to assume $K$ to be the cylinder
$$
K:=Q(R^2,2R), \quad R>0.
$$
For the initial cylinder $K$, let
\[
\mu^-\le\essinf_{K} u \qquad \text{and} \qquad \mu^+\ge\esssup_{K} u,
\]
and define
\[
\omega:=\mu^+-\mu^-.
\]
We choose $\mu^-$ small enough so that
\begin{equation}\label{bounds_inf}
\mu^-\leq \frac{\omega}{4}
\end{equation}
holds. We may assume that $\omega>0$, because otherwise there is nothing to prove. We also assume that, without loss of generality, $\omega\leq1$.

For a suitably chosen sequence $\{Q^i\}$ of cylinders, we will construct a nondecreasing sequence $\{\mu_i^-\}$ and a nonincreasing sequence $\{\mu_i^+\}$ so that
\begin{equation}\label{eq:u between mu(i)pm}
\mu_i^-\le\essinf_{Q^i} u \qquad \text{and} \qquad \mu_i^+\geq \esssup_{Q^i}u.
\end{equation}
and
\[
\omega_i :=\mu_i^+-\mu_i^-=\sigma^i\omega, \quad i=0,1,\ldots
\]
for some $\sigma \in (0,1)$. As a consequence of~\eqref{eq:u between mu(i)pm}, we have
\[
\essosc_{Q^i}u\le \omega_i.
\]
The actual proof will proceed by induction. We will assume that in $Q^i$ the two-sided bound~\eqref{eq:u between mu(i)pm} holds. Then we shall built the cylinder $Q^{i+1}$ in such a way that~\eqref{eq:u between mu(i)pm} is verified with $i+1$.

For $\delta >0$, sufficiently small, let
\begin{equation}\label{sol}
R_i = \delta^i R, \quad i=0,1,\ldots .
\end{equation}
We will work with a sequence $\{Q^i\}$ of nested and shrinking cylinders of the form
$$
Q^i:=Q(c_{i}R_{i}^p,R_{i}), \qquad c_{i} = \omega_i^{q-1}\left(\frac{\omega_{i}}{2^\lambda} \right)^{2-p}, \quad i=0,1,\ldots,
$$
where $\lambda>1$ is a constant to be fixed later depending only on the data.

In particular for the initial cylinder $Q^0=Q(c_0R^p, R)$ we choose $\mu_0^\pm$ as
\begin{equation}\label{first_bounds}
\mu_0^\pm=\mu^\pm
\end{equation}
which immediately gives
$$
\omega_0=\omega \quad \text{and} \quad c_0=\omega^{q-1}\left(\frac{\omega}{2^\lambda} \right)^{2-p}.
$$
We assume that
\begin{equation}\label{first_it}
R<2^{-\lambda}\omega^{1+\frac{1-q}{p-2}}.
\end{equation}
Indeed, otherwise we would have $\omega\leq 2^{\lambda\beta}R^{\beta}$, with $\beta=[1+(1-q)/(p-2)]^{-1}$, but then the oscillation is comparable to the radius and there is nothing to prove. Assumption \eqref{first_it} guarantees the inclusion
$$
Q(c_0R^p, R)\subset Q(R^2,2R)
$$
and thus implies the starting relation
$$
\essosc_{Q^0} u \leq \omega_0:=\omega.
$$
In order to assure the inclusion $Q^{i+1} \subset Q^i$, for all $i \geq0$, it is enough to assume $\delta \leq \sigma^{(p-q-1)/p}$. Along the proof of Theorem \ref{impliesholder} we will determine this parameter more precisely, we'll take it as $\delta := \sigma^{(p-q-1)/p}2^{-\l(3+(\lambda-1)(p-2)/p\r)}$ depending only on the data as $\lambda$ will be determined in terms of the data formerly.

Next we derive the fundamental estimates to prove our regularity result, namely Theorem \ref{main_theorem}, starting with the energy estimates.

Reasoning as in \cite{KLSU} and~\cite{KuusSiljUrba10}, we introduce the auxiliary function
\begin{align*}
\mathcal{J}((u{-}k)_\pm)= & \pm \int_{k^{q}}^{u^{q}} \left(\xi^{1/q} {-} k \right)_\pm \, d\xi
\\ = & \pm q\,\int_{k}^{u} \left(\xi{-}k \right)_\pm \xi^{q-1} \, d\xi
\\ = & q\,\int_0^{(u{-}k)_\pm}(k\pm \xi)^{q-1}\xi \, d\xi
\end{align*}
for which
\begin{equation}\label{J_derivative}
\frac{\partial }{\partial t}\mathcal{J}((u{-}k)_\pm)
=\pm\frac{\partial (u^{q})}{\partial t}(u{-}k)_\pm.
\end{equation}

We will deduce the fundamental energy estimates over the cylinders
$\Omega\times(t_1,t_2)\subset U\times(\tau_1,\tau_2)$, written for $\mathcal{J}((u{-}k)_\pm)$ and $(u{-}k)_\pm$.

\begin{lemma}\label{energy}
Let $u\ge 0$ be a weak solution of \eqref{equation} and let $k\ge 0$.
Then there exists a constant $C=C(p)>0$ such that
\begin{equation}\label{l:energyest}
\begin{split}
&\esssup_{t_1<t<t_2}\int_{\Omega}\mathcal{J}((u{-}k)_\pm)\varphi^p\, d\mu
+\int_{t_1}^{t_2}\int_{\Omega} |\nabla (u{-}k)_\pm\varphi|^{p}
\, d\nu \\
&\le C\int_{t_1}^{t_2}\int_{\Omega}(u{-}k)_\pm^p|\nabla \varphi|^p\, d\nu
+C\int_{t_1}^{t_2}\int_{\Omega} \mathcal{J}((u{-}k)_\pm)\varphi^{p-1}\left(\frac{\partial \varphi}{\partial t}\right)_+\, d\nu , \\
\end{split}
\end{equation}
for every nonnegative $\varphi\in C_0^\infty(\Omega\times(t_1,t_2))$.
\end{lemma}

\begin{proof}

We take $\eta=\pm\left(\bar u_h-k\right)_{\pm}\varphi^{p}$, where $\bar u_h=[(u^{q})_{h}]^{1/q}$, as a test function in \eqref{ws_a}, which is in the Sobolev space
$H^{1,p}_0(\Omega)$ and only admissible after an approximation, and integrate in time over $(t_1,t)$, for $t\in(t_1,t_2)$ to obtain
$$
0=\int\limits_{t_1}^{t}\int\limits_{\Omega} \! \left[\pdo {((u^{q})_{h})}{t}\,\eta
+ \left(\abs{\nabla u}^{p-2}\nabla u\right)_{h}\cdot\nabla \eta\right]\,d\nu.
$$
The first term on the right hand side can be estimated from below by using~\eqref{J_derivative},
integrating by parts and finally letting $h\rightarrow 0$ as
\begin{align*}
&\int\limits_{t_1}^{t}\int\limits_{\Omega} \! \pdo {(\bar u_h)^q}{t}\left[\pm(\bar u_{h}-k)_{\pm}\varphi^{p}\right] \,d\nu
=\int\limits_{t_1}^{t}\int\limits_{\Omega} \! \parts{}{t}\left( \mathcal{J}\left((\bar u_{h}-k)_{\pm}\right)\right)\varphi^{p} \,d\nu\\
&\longrightarrow \int\limits_{\Omega\times \{t\}} \! \mathcal{J}\left((u-k)_{\pm}\right)\varphi^{p} \,d\mu \,
-p \int\limits_{t_1}^{t}\int\limits_{\Omega} \! \mathcal{J}\left((u-k)_{\pm}\right)\varphi^{p-1}\l(\pdo{\varphi}{t}\r)_+ \,d\nu.
\end{align*}
Concerning the second integral on the right hand side we first let $h\rightarrow 0$ and then use the estimate
\begin{align*}
& |\nabla (u-k)_\pm|^{p-2}(\pm \nabla (u-k)_\pm)
\cdot  \nabla (\pm (u-k)_{\pm}\varphi^p )
\\ & \qquad \geq |\nabla (u-k)_\pm|^p \varphi^p - p |\nabla (u-k)_\pm|^{p-1}\varphi^{p-1} (u-k)_{\pm} |\nabla\varphi|.
\end{align*}
The last term is estimated further by Young's inequality as
\begin{align*}
& - p |\nabla (u{-}k)_\pm|^{p-1}\varphi^{p-1} (u{-}k)_{\pm} |\nabla\varphi|
\\ & \qquad
\geq - \frac12 |\nabla (u{-}k)_\pm|^p \varphi^p - C(u{-}k)_{\pm}^p |\nabla\varphi|^p.
\end{align*}
Hence the estimation of the second integral on the right hand side reads as
\begin{align*}
&\int\limits_{t_1}^{t}\int\limits_{\Omega} \! \left(\abs{\nabla u}^{p-2}\nabla u\right)_{h}\cdot\nabla \eta \,d\nu\\
&\geq \frac{1}{2}\int\limits_{t_1}^{t}\int\limits_{\Omega}\abs{\nabla(u-k)_{\pm}\varphi}^{p} \,d\nu-
C(p)\int\limits_{t_1}^{t}\int\limits_{\Omega} \! (u-k)_{\pm}^{p}\abs{\nabla\varphi}^{p}\,d\nu.
\end{align*}

Since $t\in(t_1,t_2)$ is arbitrary we can combine both estimates to obtain~\eqref{l:energyest}.

\end{proof}

Whenever energy estimates are used we will need to get upper and lower bounds for the auxiliary function $\mathcal{J}((u{-}k)_\pm)$. In what follows, we present the plus and minus cases separately.

As for the upper case we have
\begin{equation}\label{plus_upper_estimate}
\begin{split}
\mathcal{J}((u{-}k)_+) &= q\,\int_0^{(u{-}k)_+}(k+ \xi)^{q-1}\xi \, d\xi\\
&\le q\,k^{q-1}\int_0^{(u-k)_+} \xi d\xi \\
&= qk^{q-1}\frac{(u{-}k)_+^2}{2}
\end{split}
\end{equation}
and for the lower case we get
\begin{equation}\label{plus_lower_estimate}
\begin{split}
\mathcal{J}((u{-}k)_+)
&\ge q\,u^{q-1}\int_0^{(u{-}k)_+}\xi\, d\xi \\
&= q\,u^{q-1}\frac{(u{-}k)_+^2}{2}.
\end{split}
\end{equation}

Observe that in this case the upper and lower bounds obtained are exactly the same as the ones in \cite{KLSU} for $q=p-1$. This meets our expectations since both in \eqref{equation} and in the equation treated in \cite{KLSU} the term related to the time derivative has a negative exponent. Note further that for any strictly positive level $k$ both bounds are applicable because $u$ is above that level in the set $\brc{(u{-}k)_+>0}$.

On the other hand, for the minus case we obtain the following estimates
\begin{equation}\label{minus_upper_estimate}
\begin{split}
\mathcal{J}((u{-}k)_-) &=q\,\int_0^{(u{-}k)_-}(k-\xi)^{q-1}\xi \, d\xi\\
&\le q\,(u{-}k)_-\int_0^{(u{-}k)_-}(k-\xi)^{q-1}\, d\xi \\
&\le k^q\,(u{-}k)_-
\end{split}
\end{equation}
and
\begin{equation}\label{minus_lower_estimate}
\begin{split}
\mathcal{J}((u{-}k)_-)&\ge q\,k^{q-1}\int_0^{(u{-}k)_-}\xi\, d\xi \\
&= q\,k^{q-1}\frac{(u{-}k)_-^2}{2}.
\end{split}
\end{equation}
However, in the minus case we don't get the same estimate as in \cite{KLSU} for the upper case. Actually, we still can get the bound obtained in \cite{KLSU} (see Subsection~\ref{subsec3}) however, it is not the adequate one to use since it blows up at points where $u=0$. Note that in \cite{KLSU}, this very case is treated by virtue of Harnack's inequality which is not available in our context. As a consequence, when working with inequalities obtained via energy estimates we will need to deal simultaneously with the powers $1$, $2$ and $p$ of the truncated function $(u{-}k)_-$.

The energy estimates will be used to prove some measure estimates for distribution sets which imply an oscillation reduction in a subcylinder. Afterwards we shall need to forward this information in time since this particular subcylinder does not necessarily contain the origin. This forwarding argument will be realized by logarithmic estimates.

For this purpose we introduce the function
\[
\psi_\pm(u):=\Psi(H_k^\pm, (u{-}k)_\pm, c)=\l(\ln\l(\frac{H_k^\pm}{c+H_k^\pm-(u{-}k)_\pm}\r)\r)_+.
\]
from which we obtain the following logarithmic integral inequalities - these inequalities will be derived in a formal fashion; an accurate justification involves the use of Steklov averages (see ~\cite[p. 101--102]{DiBe93}).

\begin{lemma}\label{Harnack_logarithm}
Let $u\ge 0$ be a weak solution of equation~\eqref{equation}. Then there exists a constant $C=C(p,q)>0$ such that, for $p>2$ and $0<q<1$, we have
\begin{align*}
k^{q-1}\esssup_{t_1<t<t_2}&\,\int_{\Omega} \psi_{-}^{2}(u)(x,t)\varphi^p(x) \, d\mu \le \int_{\Omega} u^{q-1}\psi_{-}^2(u)(x,t_1)\varphi^p(x) \, d\mu \\
&\qquad \qquad \qquad +C \int_{t_1}^{t_2}\int_\Omega \psi_-(u)|(\psi_-)^{'}(u)|^{2-p}|\nabla  \varphi|^p \, d\nu\\
\end{align*}
and
\begin{align*}
\esssup_{t_1<t<t_2}&\int_{\Omega} \psi_{+}^{2}(v)(x,t)\varphi^p(x) \, d\mu \le \int_{\Omega} \psi_{+}^2(v)(x,t_1)\varphi^p(x) \, d\mu \\
&\qquad \qquad +C \int_{t_1}^{t_2}\int_\Omega v^{\frac{(1-q)(p-1)}{q}}\psi_+(v)|(\psi_+)^{'}(v)|^{2-p}|\nabla  \varphi|^p \, d\nu,
\end{align*}
where $\varphi\in C^{\infty}_0(\Omega)$ is any time-independent test function and $v=u^q$. In the plus case we assume that $u$ is strictly away from zero.
\end{lemma}

\begin{proof}
We start with the minus case. Choose
\[
\eta_-(u)=\frac{\partial}{\partial u}(\psi_-^2(u))\varphi^p
\]
in the definition of weak solution and integrate in time over $(t_1,t)$ for $t\in(t_1,t_2)$. Observe that
\begin{equation}\label{psi''}
(\psi_-^2)''=2(1+\psi_-)(\psi_-')^2.
\end{equation}
The estimate of the parabolic term reads as

\begin{align}\label{eq:log_est_1}
\int_{t_1}^{t}\int_\Omega \frac{\partial u^q}{\partial t}  \eta_{-}(u)\, d\nu &= \int_{t_1}^{t}\int_\Omega \frac{\partial }{\partial t} \int_{k^{q}}^{u^{q}} \eta_{-}(s^{1/q})\,ds \, d\nu \nonumber\\
&=\l[\int_\Omega \int_{k^{q}}^{u^{q}}\eta_{-}(s^{1/q})\,ds \, d\mu\r]_{t_1}^{t} \nonumber\\
&=\l[ \int_\Omega q \int_k^u \eta_{-}(s)s^{q-1}\,ds\, d\mu\r]_{t_1}^{t}.
\end{align}
Integrating by parts we obtain

\begin{align*}
\int_k^u \eta_{-}(s)s^{q-1}\,ds &= \int_k^u (\psi_{-}^{2}(s))^{'}s^{q-1}\,ds\,\varphi^p \\
&=\l[\psi_{-}^{2}(s)s^{q-1}\r]_{k}^{u} \varphi^p- (q-1)\int_k^u \psi_{-}^{2}(s)s^{q-2}\,ds\,\varphi^p\\
&=\psi_{-}^{2}(u)u^{q-1}\varphi^p - (1-q)\int_u^k \psi_{-}^{2}(s)s^{q-2}\,ds\,\varphi^p.
\end{align*}
From the above equality we get
$$
\int_k^u \eta_{-}(s)s^{q-1}\,ds\geq \psi_-^2(u)k^{q-1}\varphi^p
$$
and
$$
\int_k^u \eta_{-}(s)s^{q-1}\,ds\leq \psi_-^2(u)u^{q-1}\varphi^p.
$$
Using these estimates in~\eqref{eq:log_est_1} gives
\begin{align}\label{eq:log_est_2}
\int_{t_1}^{t}\int_\Omega \frac{\partial u^q}{\partial t}  \eta_{-}(u)\, d\nu \geq & q k^{q-1}\int_{\Omega} \psi_{-}^2(u)(x,t)\varphi^p(x)\,d\mu \nonumber \\
& \qquad - q \int_{\Omega} u^{q-1}\psi_{-}^2(u)(x,t_1)\varphi^p(x)\,d\mu,
\end{align}
for all $t\in(t_1,t_2)$.

Concerning the remaining term, by using~\eqref{psi''} together with Young's inequality, we obtain
\begin{align*}
&  |\nabla u|^{p-2}\nabla u \cdot \nabla \eta_-
 =  |\nabla u|^{p-2}\nabla u \cdot \nabla((\psi_-^2(u))'\varphi^p)
\\ & \qquad \geq 2 |\nabla u|^{p}(1+\psi_-)(\psi_-')^2\varphi^p
-2p|\nabla u|^{p-1}  \psi_-|\psi_-'|\varphi^{p-1}|\nabla\varphi|
\\ & \qquad \geq   |\nabla u|^{p}(\psi_-')^2\varphi^p
- C\psi_-|\psi_-'|^{2-p}|\nabla\varphi|^p,
\end{align*}
almost everywhere. The claim for the minus case follows from the previous estimate and \eqref{eq:log_est_2}.

In the plus case we choose
\[
\eta_+(v)=\frac{\partial}{\partial v}(\psi_+^2(v))\varphi^p,
\]
where $v=u^q$. Notice that~\eqref{psi''} continues to hold for this choice of test function.
After integrating in time over $(t_1,t)$ for $t\in(t_1,t_2)$, the estimate of the parabolic term follows as
\begin{align*}
&\int_{t_1}^{t}\int_\Omega \frac{\partial v}{\partial t}  \eta_+(u^q)\, d\nu= \int_{t_1}^{t}\int_\Omega \frac{\partial v}{\partial t} \frac{\partial}{\partial v}(\psi_+^2(v))\varphi^p \, d\nu\\
&=\int_{t_1}^{t}\int_\Omega \frac{\partial}{\partial t}(\psi_+^2(v))\varphi^p \, d\nu\\
&=\l[ \int_\Omega \psi_+^2(v)\varphi^p d\mu\r]_{t_1}^{t}.
\end{align*}
By using~\eqref{psi''} together with Young's inequality, we obtain
\begin{align*}
|\nabla u|^{p-2}\nabla u \cdot \nabla \eta_+ &=  |\nabla v^{1/q}|^{p-2}\nabla v^{1/q} \cdot \nabla((\psi_+^2(v))'\varphi^p)\\
& = \frac{1}{q^{p-1}}v^{\frac{(1-q)(p-1)}{q}}|\nabla v|^{p} \brc{2(1+\psi_+)(\psi_+')^2\varphi^p}\\
&\qquad +\frac{p}{q^{p-1}}v^{\frac{(1-q)(p-1)}{q}}|\nabla v|^{p-2}\nabla v \brc{2\psi_+\psi_+'\nabla \varphi} \varphi^{p-1}\\
& \geq - C(p,q)v^{\frac{(1-q)(p-1)}{q}}\psi_+(v)(\psi_+'(v))^{2-p}|\nabla\varphi|^p \varphi^{p-1}
\end{align*}
almost everywhere, from which the claim follows. Note that here we also use the fact that $f(x)=x^{1/q}$, $0<q<1$, is Lipschitz for strictly positive values of $x$. Since in this case we assume that $u$ is strictly positive by the chain rule we get
$$
\nabla v^{1/q}=\frac{1}{q}v^{\frac{1-q}{q}} \nabla v.
$$

\end{proof}

\begin{remark}
The necessity of working with a logarithmic estimate written for $u^q$ instead of $u$ itself in the plus case is intrinsically related to Lemma \ref{logarithmic_bound}.
\end{remark}


\section{Continuity analysis}

Suppose that $u\geq 0$ is a weak solution of \eqref{equation} in $Q^i=Q\left(c_i R_i^{p},R_i\right)$ and we have
\begin{equation}\label{inf_sup}
\mu_i^- \leq \essinf_{Q^i}{u} \leq \esssup_{Q^i}{u} \leq \mu_i^+, \qquad \mu_i^+-\mu_i^-=:\omega_i.
\end{equation}
Assume that $\mu_i^\pm$ satisfy
\begin{equation}\label{eq:Case I test}
\qquad 4\mu_{i}^- \leq \omega_{i} = \mu_i^+ - \mu_i^-.
\end{equation}
The case where this assumption fails will be treated in Subsection~\ref{subsec3}.

Inside $Q^i$, we consider subcylinders of smaller size
$$
Q_{t^*}\left(d_{i}R_i^{p},R_i\right), \qquad d_{i}= \omega_i^{q-1}\left(\frac{\omega_{i}}{2}\right)^{2-p}.
$$
These cylinders are contained in $Q^i$ whenever
\begin{equation}\label{time_level}
[2^{p-2}-2^{\lambda(p-2)}]R_i^p/\omega_{i}^{(p-1-q)}<t^*<0.
\end{equation}
Since $\lambda$ can always be arranged so that $N_i=c_i/d_i=2^{(\lambda-1)(p-2)}$ to be an integer, $Q^i$ can be regarded as the union, up to a set
of measure zero, of $N_i$ disjoint cylinders each congruent to $Q_{t^*}\left(d_{i}R^{p},R\right)$.

The reduction of the oscillation is based on the analysis of an alternative (see~\cite{DiBe93, Urba08}). For a constant $\alpha_{0}\in(0,1)$, that will be determined depending only on the data, either
\\

\noindent {\bf The First Alternative.}

\noindent there is a cylinder of the type $Q_{t^*}(d_iR_i^p, R_i)$
\begin{equation}\label{l:alt1}
\frac{\nu\left(\left\lbrace(x,t)\in Q_{t^*}(d_iR_i^p, R_i):\,u(x,t)<\mu^{-}_{i}
+\omega_{i}/2\right\rbrace\right)}{\nu\left( Q_{t^*}\left( d_iR_i^{p},R_i\right)\right)} \leq \alpha_{0},
\end{equation}
or, since $\mu^{+}_{i}-\frac{\omega_i}{2}=\mu^{-}_{i}+\frac{\omega_i}{2}$
\\

\noindent {\bf The Second Alternative.}

\noindent for every cylinder $Q_{t^*}(d_iR_i^p, R_i)\subset Q^i$, we have
\begin{equation}\label{l:alt2}
\frac{\nu\left(\left\lbrace(x,t)\in Q_{t^*}(d_iR_i^p, R_i):\,u(x,t)>\mu^{+}_{i}
-\omega_{i}/2\right\rbrace\right)}{\nu\left(Q_{t^*}\left(d_iR_i^p, R_i\right)\right)}< 1-\alpha_{0}.
\end{equation}

\noindent In both cases, we will prove that the essential oscillation of $u$ within a smaller cylinder
decreases in a measurable way, measurable so that we can derive a modulus of continuity. The constant $\alpha_0$ will be fixed in the course of the proof of Lemma~\ref{main_lemma1}.

\subsection{Reduction of the oscillation in the first alternative}

Now we assume that the first alternative holds.

\begin{lemma}\label{main_lemma1}
There exists a constant $\alpha_0\in (0,1)$, depending only on the data, such that if \eqref{l:alt1} holds for some $t^*$ as in~\eqref{time_level}, then
$$
u(x,t)>\mu_i^-+\frac{\omega_i}{4} \qquad \text{a.e.} \quad \text{in}\quad \tfrac{1}{2}Q_{t^*}\l(d_iR_i^p,R_i\r).
$$
\begin{proof}
Define
\[
R_n=\frac{R_i}{2}+\frac{R_i}{2^{n+1}},
\qquad Q_{n}=B_{n}\times T_{n}=B(R_n)\times (t^*-d_i R_n^p,t^*)
\]
and
\[
k_n = \mu_i^- + \frac{\omega_i}{4}\left(1 + \frac1{2^{n}}  \right)
\]
for $n=0,1,\dots$. Observe that
\be\label{bound_level}
\frac{\omega_i}{4}\leq k_n \leq \omega_i.
\ee
Here the first inequality is trivial and follows from the definition of $k_n$ while the second one follows by assumption~\eqref{eq:Case I test}.

Choose cutoff functions $\varphi_n \in C^\infty(Q_n)$, vanishing on the parabolic boundary of $Q_n$, and such that $0\le\varphi_n\le 1$,
$\varphi_n=1$ in $Q_{n+1}$,
\begin{align}\label{gradient_estimates}
|\nabla \varphi_n|\le\frac{C2^{n}}{ R_i}
\quad\text{and}\quad \left(\frac{\partial \varphi_n}{\partial t}\right)_+\le \frac{C2^{pn}}{d_i R_i^p}.
\end{align}

By H\"older's inequality and Sobolev's inequality \eqref{Sobolevzero} together with the doubling property of the measure, we have
\begin{equation}\label{HoldSobo temp}
\begin{split}
\dashint_{Q_{n+1}} & (u{-}k_n)_-^{2(1-p/\kappa)+ p} \, d\nu
\\ \leq &
\frac{\nu(Q_{n})}{\nu(Q_{n+1})}
\dashint_{Q_{n}}  (u{-}k_n)_-^{2(1-p/\kappa)+ p} \varphi_n^{p(1-p/\kappa)+ p} \, d\nu \\
\leq & C \dashint_{T_{n}}\left(\dashint_{B_{n}} (u{-}k_n)_-^2\varphi_n^p\, d\mu\right)^{1-p/\kappa}
\left(\dashint_{B_{n}}((u{-}k_n)_- \varphi_n)^{\kappa}\, d\mu\right)^{p/\kappa}\, dt \\
\leq & C R_{n}^{p}  \left(\esssup_{T_n} \dashint_{B_{n}} (u{-}k_n)_-^2\varphi_n^p \, d\mu\right)^{1-p/\kappa}
\dashint_{Q_n} |\nabla ((u{-}k_n)_-\varphi_n)|^p\,d\nu,
\end{split}
\end{equation}
where $\kappa$ is the Sobolev exponent as in \eqref{kappa}. Notice that this estimate continues to hold for any function given in the appropriate Sobolev space and it is independent of the choice of $Q_n$, modulo a constant.

Taking into account \eqref{minus_upper_estimate} and \eqref{minus_lower_estimate} energy estimate \eqref{l:energyest} in the minus case yields
\begin{align*}
& \qquad k_n^{q-1}\esssup_{T_n}\dashint\limits_{B_n} \! (u{-}k_n)_-^{2}\varphi_n^{p}\,d\mu \,
+\, d_iR_n^p\dashint\limits_{Q_n}\abs{\nabla(u{-}k_n)_-\varphi_n}^{p}\,d\nu \nonumber \\
&\leq C\,d_iR_n^p\dashint\limits_{Q_n}(u{-}k_n)_-^{p}\abs{\nabla \varphi_n}^{p}\,d\nu +C\,d_iR_n^p\, k_n^{q}\dashint\limits_{Q_n}(u{-}k_n)_-\varphi_n^{p-1}
\left(\frac{\partial \varphi_n}{\partial t}\right)_+\,d\nu.
\end{align*}

Using \eqref{bound_level} and \eqref{gradient_estimates} together with the fact that $(u{-}k_n)_- \leq \omega_i$ almost everywhere the above inequality can be rewritten as
\begin{align*}
& \omega_i^{q-1}\esssup_{T_n}\dashint\limits_{B_n} \! (u{-}k_n)_-^{2}\varphi_n^{p}\,d\mu \,
+\, d_iR_n^p\dashint\limits_{Q_n}\abs{\nabla(u{-}k_n)_-\varphi_n}^{p}\,d\nu \nonumber \\
&\leq C2^{np} d_i\brc{\l(\frac{\omega_i}{2}\r)^p + \omega_i^q \frac{\omega_i}{2}\frac{1}{d_i}}\frac{\nu(A_n)}{\nu(Q_n)},
\end{align*}
where
\[
A_{n}=\left\{(x,t)\in Q_{n}:u(x,t)< k_n\right\}.
\]
Thus, by using the definition of $d_i$, we conclude with
\begin{equation}\label{ml_1}
\esssup_{T_n}  \dashint_{B_{n}} (u{-}k_n)_-^2\varphi_n^p \, d\mu \leq
C 2^{np} \l(\frac{\omega_i}{2}\r)^2 \frac{\nu(A_n)}{\nu(Q_n)}
\end{equation}
and
\begin{equation}\label{ml_2}
R_n^p\dashint\limits_{Q_n}\abs{\nabla(u{-}k_n)_-\varphi}^{p}\,d\nu \leq
C 2^{np} \l(\frac{\omega_i}{2}\r)^p \frac{\nu(A_n)}{\nu(Q_n)}.
\end{equation}
To obtain an algebraic estimate, we need to estimate the left hand side of \eqref{HoldSobo temp} from below:
\begin{equation}\label{ml3}
\begin{split}
&\dashint_{Q_{n+1}}(u-k_{n})_-^{2(1-p/\kappa)+ p} \, d\nu   \\
& \quad \geq \dashint_{Q_{n+1}}(u-k_{n})_-^{2(1-p/\kappa)+ p} \chi_{\{(u-k_{n+1})_->0\}} \, d\nu \\
& \quad \geq |k_{n+1}-k_n|^{2(1-p/\kappa)+ p}\int_{Q_{n+1}} \chi_{\{(u-k_{n+1})_->0\}}\, d\nu \\
& \quad \geq  \l(\frac{\omega_i}{2^{n+3}}\r)^{2(1-p/\kappa)+ p}\frac{\nu(A_{n+1})}{\nu(Q_{n+1})}.
\end{split}
\end{equation}
Inserting \eqref{ml_1}, \eqref{ml_2} and \eqref{ml3} into \eqref{HoldSobo temp} arrive at
\begin{equation}\label{main_lemma_1}
\frac{\nu(A_{n+1})}{\nu(Q_{n+1})}\leq C^{n+1}\left(\frac{\nu(A_{n})}{\nu(Q_{n})}\right)^{2-p/\kappa}, \quad n=0,1,2,\dots.
\end{equation}
By setting
\[
Y_n = \frac{\nu(A_n)}{\nu(Q_n)}
\]
we obtain the recursive relation
\[
Y_{n+1} \le C^{n+1} Y_n^{2-p/\kappa},
\]
for some constant $C$ depending only on the data. We conclude, using Lemma~\ref{geometric_convergence}, that if
$$
Y_0\leq C^{-1/(1-p/\kappa)+1-(1-p/\kappa)^2}:=\alpha_0
$$
holds then $Y_n \to 0$, as $n\to \infty$. But since this condition is precisely the assumption of the first alternative \eqref{l:alt1} for the indicated choice of $\alpha_0$, the result follows.
\end{proof}
\end{lemma}

Now we make use of the previous lemma to prove that the set where $u$ is close to
its infimum can be made arbitrarily small within the 
time interval $-\theta<t<0$, where
\begin{equation}\label{theta}
-\theta:=t^*-d_i\l(\frac{R_i}{2}\r)^p
\end{equation}
for some $t^*$ as in \eqref{time_level}.
\begin{lemma}\label{forw_time}
Assume \eqref{l:alt1} holds for some $t^*$ satisfying \eqref{time_level} and let $\theta$ be given as in \eqref{theta}.
For every number $\alpha_1\in(0,1)$, there exists $s_*\in\mathbb{N}$, depending only on the data, such that
$$
\mu\l(\brc{x\in \tfrac{1}{4}B\l(R_i\r)\,:\,u(x,t)<\mu_i^-+\frac{\omega_i}{2^{s_*}}}\r)
\leq \alpha_1\mu\l({\tfrac{1}{4}B\l(R_i\r)}\r), \quad \forall t\in (-\theta,0).
$$
\begin{proof}
In Lemma~\ref{Harnack_logarithm} consider the estimate written for the truncated function $(u-k)_{-}$  over the cylinder $(1/2)Q\l(\theta R_i^p,R_i\r)$  with
$$
k=\mu_i^-+\frac{\omega_i}{4},\qquad c=\frac{\omega_i}{2^{n+2}}\qquad\text{and} \qquad H_k^-=\frac{\omega_i}{4}.
$$
Here $n\in \mathbb{N}$ will be determined later depending on the data.
Observe that, since $u(x,-\theta)>k$ in $(1/2)B\l(R_i\r)$ by Lemma \ref{main_lemma1}, we have
$$
\psi_-(u)(x,-\theta)=0, \quad x\in \tfrac{1}{2}B\l(R_i\r).
$$
On the other hand, since $(u-k)_-\leq\omega_i/4$ we get
$$
\psi_-(u)\leq\ln{\l(\frac{2^{-2}\omega_i}{2^{-(n+2)}\omega_i}\r)}=n\ln 2
$$
and
$$
\abs{(\psi_-)^{'}(u)}^{2-p}=\l(H_k^- -(u-k)_- +c\r)^{p-2}\leq\l(\frac{\omega_i}{2}\r)^{p-2}.
$$
Take cutoff function $\varphi\in C^{\infty}_0\l((1/2)B\l(R_i\r)\r)$ independent of time with properties $0\leq\varphi\leq1$,
$\varphi=1\in (1/4)B\l(R_i\r)$ and
$$
\abs{\nabla\varphi}\leq\frac{C}{R_i}. $$
In the set
$$
\brc{x\in\tfrac{1}{4}B\l(R_i\r)\,:\,u<\mu_i^-+\frac{\omega_i}{2^{n+2}}}
$$
$\psi_-(u)$ can be bounded from below as
$$
\psi_-(u)\geq\ln {\l(\frac{2^{-2}\omega_i}{2^{-(n+2)}\omega_i+2^{-(n+2)}\omega_i}\r)}\geq (n-1)\ln 2.
$$

Gathering these estimates all together and recalling that $k\leq\omega_i/2$, due to \eqref{eq:Case I test}, we arrive at
\begin{equation*}
\begin{split}
\omega_i^{q-1}(n-1)^2&(\ln 2)^2 \mu\l(\brc{x\in \tfrac{1}{4}B\l(R_i\r)\,:\,u<\mu_i^-+\frac{\omega_i}{2^{n+2}}}\r)\\
&\leq k^{q-1}\esssup_{-\theta<t<0} \int_{\frac{1}{2}B\l(R_i\r)} \psi_-^2(u)\varphi\,d\mu\\
&\leq C \int_{-\theta}^{0}\int_{\tfrac{1}{2}B\l(R_i\r)} \psi_-(u)\abs{(\psi_-)^{'}(u)}^{2-p}\abs{\nabla\varphi}^p\,d\nu\\
& \leq C\,n\,(\ln 2) \frac{1}{R_i^p}\theta\mu\l(\tfrac{1}{2}B\l(R_i\r)\r) \\
& \leq C\,n\,(\ln 2)\l(\frac{1}{2^\lambda}\r)^{2-p}\omega_i^{q-1}\mu\l(\tfrac{1}{4}B\l(R_i\r)\r).
\end{split}
\end{equation*}
In the last inequality we used the fact $\theta\leq c_iR_i^p$ and the doubling property of the measure. A simplification in the above inequality gives
\begin{equation*}
\mu\l(\brc{x\in \tfrac{1}{4}B\l(R_i\r)\,:\,u<\mu_i^-+\frac{\omega_i}{2^{n+2}}}\r)\leq C\,\frac{n}{(n-1)^2}2^{\lambda(p-2)}\mu\l(\tfrac{1}{4}B\l(R_i\r)\r).
\end{equation*}
To conclude, we choose $s_*=n+2$
with $n>1+\frac{2C}{\alpha_1}2^{\lambda(p-2)}$.
\end{proof}
\end{lemma}

The information of this lemma will be used to show that an estimate similar to the conclusion of Lemma \ref{main_lemma1} holds in a full cylinder that includes the origin.
\begin{lemma}
Assume \eqref{l:alt1} holds for some $t^*$ satisfying \eqref{time_level} and let $\theta$ be given as in \eqref{theta}. There exists $s_1\in\mathbb{N}$, depending only on the data, such that
$$
u(x,t)>\mu_i^- + \frac{\omega_i}{2^{s_1+1}}, \qquad \text{a.e.} \quad \text{in} \quad Q\l(\theta,\frac{R_i}{8}\r).
$$
\begin{proof}
Let
\[
R_n=\frac{R_i}{8}+\frac{R_i}{2^{n+3}},
\qquad Q_{n}=Q\l(\theta,R_n\r)
\]
and
\[
k_n = \mu_i^- + \frac{\omega_i}{2^{s_1+1}}\left(1 + \frac1{2^{n}}  \right)
\]
for $n=0,1,\dots$. Take cutoff functions $\varphi_n(x) \in C_0^\infty(B_n)$, where $B_n=B(R_n)$, vanishing on the boundary of $B_n$, and such that $0\le\varphi_n\le 1$,
$\varphi_n=1$ in $B_{n+1}$,
\begin{align*}
|\nabla \varphi_n|\le\frac{C2^{n}}{ R_i}.
\end{align*}
Observe that $k_n\leq\omega_i$ by assumption~\eqref{eq:Case I test} and $u(x,-\theta)>\mu_i+\omega_i/4\geq k_n$ in $B_n\subset (1/2)B(R_i)$ by Lemma~\ref{main_lemma1} which implies $(u(x,-\theta)-k_n)_-=0$ in $B_n$.
Using these estimates together with \eqref{minus_upper_estimate} and \eqref{minus_lower_estimate} in the energy estimate \eqref{l:energyest} written for $(u-k_n)_-$ we arrive at
\begin{align*}\label{eq:energy_l2}
&\omega_i^{q-1}\esssup_{-\theta<t<0}\dashint\limits_{B_n} \! (u{-}k_n)_-^{2}\varphi_n^{p}\,d\mu \,
+\, \theta\dashint\limits_{Q_n}\abs{\nabla(u{-}k_n)_-\varphi}^{p}\,d\nu \nonumber \\
&\leq C\,\theta\dashint\limits_{Q_n}(u{-}k_n)_-^{p}\abs{\nabla \varphi_n}^{p}\,d\nu.
\end{align*}
We estimate further
$$
(u-k_n)_-\leq\frac{\omega_i}{2^{s_1}} \quad \text{and} \quad \theta\leq c_iR_i^p
$$
and invoke the estimate on $|\nabla \varphi_n|$ to arrive at
$$
\esssup_{-\theta<t<0}\dashint\limits_{B_n} \! (u{-}k_n)_-^{2}\varphi_n^{p}\,d\mu\leq C\,2^{np}\l(\frac{\omega_i}{2^{s_1}}\r)^2 2^{(\lambda-s_1)(p-2)}\frac{\nu\l(A_n\r)}{\nu\l(Q_n\r)}
$$
and
$$
R_n^p\dashint\limits_{Q_n}\abs{\nabla(u{-}k_n)_-\varphi}^{p}\,d\nu\leq  C\,2^{np}\l(\frac{\omega_i}{2^{s_1}}\r)^p\frac{\nu\l(A_n\r)}{\nu\l(Q_n\r)},
$$
where $A_n=\left\{(x,t)\in Q_{n}:u(x,t)< k_n\right\}$.
Note that estimate~\eqref{HoldSobo temp} continues to hold in the setting of this lemma and, by the same kind of reasoning as in the proof of Lemma \ref{main_lemma1}, the left hand side can be estimated from below by
$$
\dashint_{Q_{n+1}}(u-k_{n})_-^{2(1-p/\kappa)+ p} \, d\nu \geq  \l(\frac{\omega_i}{2^{s_1+n+2}}\r)^{2(1-p/\kappa)+ p}\frac{\nu(A_{n+1})}{\nu(Q_{n+1})}.
$$
Substituting these last three estimates in~\eqref{HoldSobo temp} and assuming that $s_1>\lambda$, we obtain
$$
\frac{\nu(A_{n+1})}{\nu(Q_{n+1})}\leq C^{n+1}\l(\frac{\nu\l(A_n\r)}{\nu\l(Q_n\r)}\r)^{2-p/\kappa}.
$$
We set
\[
Y_n = \frac{\nu(A_n)}{\nu(Q_n)},
\]
as before, and rephrase the above inequality as
\[
Y_{n+1} \le C^{n+1} Y_n^{2-p/\kappa}.
\]
By defining
$$\alpha_1 = C^{-1/(1-p/\kappa)+1-(1-p/\kappa)^2}$$
we conclude, using Lemma~\ref{geometric_convergence}, that if $Y_0\leq\alpha_1$ holds then $Y_n \to 0$ as $n\to \infty$. Now we apply Lemma~\ref{forw_time} with such an $\alpha_1$ and conclude there exists $s_*=:s_1$, depending only on the data, such that
$$
\mu\l(\brc{x\in \tfrac{1}{4}B\l(R_i\r)\,:\,u(x,t)<\mu_i^-+\frac{\omega_i}{2^{s_1}}}\r)
\leq \alpha_1\mu\l({\tfrac{1}{4}B\l(R_i\r)}\r), \quad \forall t\in (-\theta,0),
$$
which is exactly $Y_0\leq\alpha_1$. The result is concluded since $Y_n \to 0$ as $n\to \infty$ implies that $A_n \to 0$ as $n\to \infty$.
\end{proof}
\end{lemma}

We finally reach the reduction of the oscillation of the weak solution in the case of the first alternative.

\begin{corollary}\label{cor:Case I conclusion}
Suppose that~\eqref{eq:u between mu(i)pm} and ~\eqref{eq:Case I test} hold in $Q^i=Q(c_iR_i^p,R_i)$. Assume also that \eqref{l:alt1} is verified. Then
there is a constant $\sigma_I \in (0,1)$, depending only on the data, such that
\begin{equation}\label{shrink_2}
    \essosc_{\frac{1}{8}Q\l(d_i R_i^p, R_i\r)}{u} \leq \sigma_I \omega_{i}.
\end{equation}
\begin{proof}
Observe that
$$
d_i \l(\frac{R_i}{8}\r)^p\leq \theta=-t^*+ d_i\l(\frac{R_i}{2}\r)^p
$$
since $t^*<0$. Therefore
$$
\essinf_{\frac{1}{8}Q\l(d_i R_i^p, R_i\r)} u
\geq \essinf_{Q\l(\theta,\frac{R_i}{8}\r)} u \geq \mu_i^-+\frac{\omega_i}{2^{s_1+1}}.
$$
Put
\[
\mu_{i+1}^+=\mu_i^+ \qquad \textrm{and} \qquad \mu_{i+1}^-=\mu_i^-+\frac{\omega_i}{2^{s_1+1}}.
\]
Then we have
$$
\essosc_{\frac{1}{8}Q\l(d_i R_i^p, R_i\r)}{u} \leq \mu_{i+1}^+ - \mu_{i+1}^- = \l(1-\frac{1}{2^{s_1+1}}\r) \omega_{i}
$$
and the corollary follows with $\sigma_I=1-\frac{1}{2^{s_1+1}}\in (3/4,1)$.
\end{proof}
\end{corollary}
\smallskip

\subsection{Reduction of the oscillation in the second alternative}

Now we analyze the second alternative. Assume~\eqref{l:alt2} holds for all cylinders of the type $Q_{t^*}\left(d_{i}R_i^{p},R_i\right)$, where $t^*$ is as in~\eqref{time_level}.
\medskip

Fix a cylinder $Q_{t^*}\left(d_{i}R_i^{p},R_i\right)$. We deduce from~\eqref{l:alt2} that there exists a time level $t^0 \in\left(t^* - d_iR_i^{p},t^*-\frac{\alpha_0}{2}d_iR_i^{p}\right)$ such that
\be\label{l:alt2ineq}
\mu\left(\left\lbrace x\in B(R_i):\,u(x,t^{0})>\mu_{i}^{+}-\omega_{i}/2\right\}\right)
\leq\left(\frac{1-\alpha_0}{1-\alpha_0/2}\right)\mu(B(R_i)).
\ee
In fact, if~\eqref{l:alt2ineq} is violated for all $\left(t^* - d_iR_i^{p},t^*-\frac{\alpha_0}{2}d_iR_i^{p}\right)$ we would get
\begin{align*}
&\nu\left( \left\lbrace (x,t) \in Q_{t^*}\left(d_iR^{p},R_i\right):\, u(x,t)>\mu^{+}_{i}-\omega_{i}/2\right\}\right) \\
&\qquad \geq \int_{t^*-d_iR_i^{p}}^{t^*-(\alpha_0/2)d_iR_i^{p}} \! \mu\left(\left\lbrace x \in B(R_i) \ :\ u(x,t)>\mu^{+}_{i}-\omega_{i}/2\right\}\right) \, dt\\
& \qquad > (1-\alpha_0)\nu(Q\left(d_iR_i^{p},R_i \right)),
\end{align*}
\noindent which contradicts~\eqref{l:alt2}.

The next lemma asserts that a similar property still holds for all time levels in an interval up to the origin.

\begin{lemma}\label{logarithmic_bound}
Assume~\eqref{l:alt2} is verified. There exists $s$, depending only upon the data, such that
\[
\mu\left(\left\{x\in B(R_i)\ : \ u(x,t) > \mu_{i}^+ - 2^{-(s+s_o+1)} \omega_{i} \right\}\right)\le\left(1-\frac{\alpha_0 }{4} \right)\mu(B(R_i)),
\]
for almost every $t\in\left(-(c_i/2)R_i^{p},0\right)$.
\begin{proof}
Set $Q:=B(R_i)\times (t^0,t^*)$ and let
\[
c=\l(\frac{\omega_{i}}{2^{s+s_o+1}}\r)^q, \qquad k = \l(\mu_{i}^+\r)^q - \l(\frac{\omega_{i}}{2^{s_o}}\r)^q
, \qquad
H_k^+= \l(\frac{\omega_{i}}{2^{s_o}}\r)^q,
\]
where $s_o$ is such that
$$ \l(\mu_{i}^+\r)^q - \l(\frac{\omega_{i}}{2^{s_o}}\r)^q > \l(\mu_{i}^+-\frac{\omega_{i}}{2}\r)^q .$$

We shall use Lemma~\ref{Harnack_logarithm} to forward the information involved in~\eqref{l:alt2ineq} in time.
Set $v=u^q$ and recall the definition
\[
\psi_+(v)=\Psi(H_k^+,(v{-}k)_+,c)=\ln^+\l(\frac{H_k^+}{c+H_k^+-(v{-}k)_+}\r).
\]
From $(v-k)_+ \leq \l(\omega_i/2^{s_o}\r)^q \equiv H_k^+$ we obtain
\[
\psi_+(v)\le \ln\frac{\l(\frac{\omega_{i}}{2}\r)^q}{\l(\frac{\omega_i}{2^{s+s_o+1}}\r)^q}= (s+s_o)q\ln 2.
\]
We continue with the estimate
\[
|\psi'_+(v)|^{2-p}\le \left( H_k^+ +c \right)^{p-2} \le C(p,q)\left(\frac{\omega_{i}}{2^{s_o}}\right)^{q(p-2)}.
\]
Denote $B:=B(R_i)$ and let $\varphi\in C_0^\infty(B)$ be a time-independent cutoff function satisfying $0\le\varphi\le 1$, $\varphi=1$ in $(1-\delta)B$ and
$$
|\nabla \varphi|\le \frac{C}{\delta R_i},
$$
where $0<\delta<1$ is to be determined later.
Define
$$
S=\brc{x\in(1-\delta)B \ : \ v>\l(\mu_i^+\r)^q-\l(\frac{\omega_i}{2^{s+s_o+1}}\r)^q}
$$
and
$$
S'=\brc{x\in(1-\delta)B \ : \ v>\l(\mu_i^+ - \frac{\omega_i}{2^{s+ s_o+1}}\r)^q}.
$$
Observe that $S'\subset S$, for $q\in(0,1)$.

In the set $S$ we get
\begin{align*}
\psi_+(v)\geq\ln \frac{\l(\frac{\omega_i}{2^{s_o}}\r)^{q}}{2\l(\frac{\omega_i}{2^{s+s_o+1}}\r)^{q}}
\geq ((s+1)q-1) \ln 2.
\end{align*}
Now we apply Lemma~\ref{Harnack_logarithm} in the plus case with these choices to conclude
\begin{equation}\notag
\begin{split}
((s+1)q-1)^2 (\ln2)^2 \mu(S) & \le\esssup_{t^0<t<t^*}\int_{B} \psi_+^2(v)(x,t)\varphi^p(x) \, d\mu \\
& \le \int_{B}\psi_+^2(v)(x,t^0)\varphi^p(x) \, d\mu \\
&\quad +C \int_{t^0}^{t^*}\int_{B}v^{\frac{(1-q)(p-1)}{q}}\psi_+(v)|(\psi_+)^\prime(v)|^{2-p}|\nabla  \varphi|^p \, d\nu \\
& \le ((s+s_o)q)^2 (\ln2)^2  \frac{1-\alpha_0}{1-\alpha_0/2}\mu(B)\\
& \quad +C \omega_i^{(1-q)(p-1)+q(p-2)}\frac{(s+s_o)q \ln 2}{\delta^pR_i^p}\abs{t^*-t^0} \mu(B),
\end{split}
\end{equation}
for almost every $t\in(t^0,t^*)$. Here in the third inequality we used~\eqref{l:alt2ineq} and
$$
v=u^q\leq (2\omega_i)^q
$$
which follows from~\eqref{eq:Case I test}. Using the facts $\abs{t^*-t^0}\leq d_iR_i^p$ and $S'\subset S$ we reach the estimate
$$
\mu(S')\leq\mu(S)\leq \frac{((s+s_o)q)^2}{((s+1)q-1)^2} \frac{1-\alpha_0}{1-\alpha_0/2}\mu(B)+C \frac{(s+s_o)q}{((s+1)q-1)^2}\frac{1}{\delta^p}\mu(B).
$$
On the other hand, by the annular decay property \eqref{annular_decay}, we have
\begin{align*}
& \mu(\{ x\in B  \ :\ v(x,t)> \big(\mu_{i}^+- 2^{-(s+s_o+1)} \omega_{i}\big)^q \}) \\
& \le\mu(B \setminus (1-\delta) B)+\mu(\{x \in (1-\delta)B) \ : \ v(x,t)> \big(\mu_{i}^+- 2^{-(s+s_o+1)}\omega_{i}\big)^q \})\\
& \le C\delta^\alpha\mu(B)+\mu(S'),
\end{align*}
for almost every $t\in(t^0,t^*)$. For the first term, we choose $\delta$ small enough so that
\[
C\delta^\alpha<\frac{\alpha_0}{24}
\]
and for the second term we use the previous estimate. By choosing $s$ large enough so that
\[
\frac{1-\alpha_0}{1-\alpha_0/2} \, \frac{((s+s_o)q)^2}{((s+1)q-1)^2}
\le 1-\frac{\alpha_0}{3}
\]
and
\[
\frac{C(s+s_o)q}{\delta^p((s+1)q-1)^2}
\le\frac{\alpha_0}{24}
\]
hold, we get
\be\label{eq:forw_conc}
\mu\l(\brc{x\in B  \ :\ v(x,t)> \l(\mu_{i}^+- 2^{-(s+s_o+1)} \omega_{i}\r)^q }\r)\leq1-\frac{\alpha_0}{4}
\ee
for almost every $t\in(t^0,t^*)$.

Since~\eqref{l:alt2ineq} holds for all cylinders of type $Q_{t^*}(d_iR_i^p,R_i)$, the conclusion~\eqref{eq:forw_conc} holds for all time levels
$$
t\geq -(c_i-d_i)R_i^p-\frac{\alpha_0}{2}d_iR_i^p.
$$
By choosing
\be\label{asp_lambda_1}
2^{(\lambda-1)(p-2)}\geq2
\ee
we get $c_i/d_i\geq 2-\alpha_0,$ which implies
$$
-(c_i-d_i)R_i^p-\frac{\alpha_0}{2}d_iR_i^p\leq-\frac{c_i}{2}R_i^p.
$$

\end{proof}
\end{lemma}

Next, by using the information of the previous lemma we will prove a critical estimate which states that, within a cylinder around the origin, the set where $u$ is close to its supremum can be made arbitrarily small.

\begin{lemma}
For every $\alpha_2\in(0,1)$ there exists $s_2\geq s+s_o$, depending only on the data and $\alpha_2$,
such that
\[
\frac{\nu\left(\brc{ (x,t) \in Q\l( (c_i/2)R_i^p, R_i\r) \ : \ u(x,t) > \mu_{i}^+ - 2^{-s_2} \omega_{i}  } \right)}{\nu\l(Q\l( (c_i/2) R_i^p,R_i\r)\r)} \leq \alpha_2.
\]
\begin{proof}

Consider the levels
\[
h=\mu_{i}^+- 2^{-(n+1)} \omega_{i}
\]
and
\[
k=\mu_{i}^+- 2^{-n} \omega_{i},
\]
where $n \geq s+s_o+1$ will be chosen large and $s$ is as in Lemma~\ref{logarithmic_bound}. Taking $B=B(R_i)$, using Lemma~\ref{logarithmic_bound} and the fact that $n\ge s+s_o+1$ we have for almost every $t\in(-(c_i/2) R_i^p,0)$
\begin{align*}
\mu(\{x\in B :  w(x,t)=0\})&=\mu(\{x\in B  :  u(x,t)\le k\})
\ge\frac{\alpha_0}{4}\mu(B),
\end{align*}
where
\[
w=
\begin{cases}
h-k,\quad &u\ge h, \\
u-k,\quad &k<u<h, \\
0, \quad &u\le k.
\end{cases}
\]
Thus, we obtain
\[
w_{B}(t)=\dashint_{B \times\{t\}} w \, d\mu \le \l(1-\frac{\alpha_0}{4}\r)(h-k)
\]
and, consequently,
\[
h-k-w_{B}(t) \ge \frac{\alpha_0}{4}(h-k),
\]
for almost every $t\in(-(c_i/2) R_i^p,0)$.
Let
\[
A_n (t) = \brc{ x \in B(R_i) \ : \ u(x,t) > \mu_{i}^+ - 2^{-n} \omega_{i}  }
\]
and
\[
A_n  = \brc{ (x,t) \in Q\l(\tfrac{c_i}{2}R_i^p, R_i\r) \ : \ u(x,t) > \mu_{i}^+ - 2^{-n} \omega_{i}  }.
\]
Using the $(q,q)$-Poincar\'e inequality for some $q<p$ (see~\eqref{poincare} and the subsequent remarks), yields
\begin{align*}
(h-k)^q\mu(A_{n+1}(t))&\le\l(\frac{4}{\alpha_0}\r)^q\int_{B\times\{t\}}|w-w_{B}(t)|^q\, d\mu \\
&\le
CR_i^q\int_{2B \times\{t\}} |\nabla w|^q \, d\mu \\
&= CR_i^q\int_{2B \times\{t\} \cap [k<u<h]} |\nabla u|^q \, d\mu,
\end{align*}
for almost every $t\in(-(c_i/2) R_i^p,0)$; and thereafter we integrate the above inequality over time to get
\[
(h-k)^q\nu(A_{n+1})\le C R_i^q\int_{ Q\l(c_iR_i^p, 2R_i\r)\cap [k<u<h] }|\nabla u|^q \varphi^p\, d\nu,
\]
$\varphi$ being a cutoff function in $ C^\infty_0\l(Q\l(c_iR_i^p, 2R_i\r)\r)$ such that
$0\le\varphi\le 1$, $\varphi=1$ in $Q\l((c_i/2)R_i^p, R_i\r)$, $\vp$ vanishes on the parabolic boundary of $Q\l(c_iR_i^p, 2R_i\r)$,
and
\[
|\nabla \varphi|\le\frac{C}{R_i} \quad\text{and}\quad \left(\frac{\partial \varphi}{\partial t}\right)_+\le \frac{C}{c_i R_i^p}.
\]
Now, H\"older's inequality and the doubling condition give
\begin{align}\label{eq:diff_sets}
& (h-k)^q\nu(A_{n+1}) \nonumber
\\ & \qquad  \le C\l(R_i^p \int_{Q(c_i R_i^p,2R_i)} |\nabla (u-k)_+|^p\varphi^p \, d\nu\r)^{q/p}\nu(A_{n}\setminus A_{n+1})^{1-q/p}.
\end{align}
The first factor on the right hand side can be estimated by Lemma~\ref{energy} written for the plus case as
\begin{equation}\nonumber
\begin{split}
&\int_{Q\l(c_iR_i^p,2R_i\r)} |\nabla (u-k)_+|^p\varphi^p \, d\nu \\
&\quad \le C\int_{Q\l(c_iR_i^p,2R_i\r)}\brc{(u-k)_+^p|\nabla\varphi |^p \,d\nu +\,k^{q-1} (u-k)_+^2 \left(\frac{\partial \varphi}{\partial t}\right)_+}\, d\nu \\
& \quad \le \frac{C}{R_i^p}\l( \l(\frac{\omega_{i}}{2^n}\r)^p+\l(\frac{\omega_{i}}{2}\r)^{q-1}\l(\frac{\omega_{i}}{2^n}\r)^2
\l(\frac{\omega_{i}}{2^{\lambda}}\r)^{p-2}\omega_i^{1-q}\r)\nu\left(Q\l(c_iR_i^p,2R_i\r)\right) \\
&\quad \le \frac{C}{R_i^p}\left(\frac{\omega_{i}}{2^n}\right)^p \l\{1+2^{(n-\lambda)(p-2)}\r\}\nu\left(Q\l(\tfrac{c_i}{2}R_i^p,R_i\r)\right)\\
&\quad \le \frac{C}{R_i^p}\left(\frac{\omega_{i}}{2^n}\right)^p \nu\left(Q\l(\tfrac{c_i}{2}R_i^p,R_i\r)\right).
\end{split}
\end{equation}
Above, in the second inequality we have used $(u-k)_+\leq \omega_i/2^n$ and $k\geq \omega_i/2$, in the third one we used the doubling property of the measure $\nu$ and in the forth one we assumed $\lambda>n$. Using this estimate in~\eqref{eq:diff_sets} we arrive at
\[
\l(\frac{\omega_{i}}{2^{n+1}}\r)^q
\nu(A_{n+1})\le C
\l(\frac{\omega_{i}}{2^n}\r)^q
\l[\nu\left(Q\l(\tfrac{c_i}{2}R_i^p,R_i\r)\right)\r]^{q/p}\nu(A_{n}\setminus A_{n+1})^{1-q/p}.
\]
Finally, summing $n$ over $s+s_o+1,\dots,s_2-1$ gives
\begin{align*}
(s_2-s-s_o)\nu(A_{s_2})^{p/(p-q)}
&\le C\l[\nu\left(Q\l(\tfrac{c_i}{2}R_i^p,R_i\r)\right)\r]^{q/(p-q)}\nu\left(Q\l(\tfrac{c_i}{2}R_i^p,R_i\r)\right)\\
&= C\l[\nu\left(Q\l(\tfrac{c_i}{2}R_i^p,R_i\r)\right)\r]^{p/(p-q)},
\end{align*}
and hence
\[
\nu(A_{s_2})\le \frac{C}{(s_2-s-s_o)^{(p-q)/p}}\nu\left(Q\l(\tfrac{c_i}{2}R_i^p,R_i\r)\right).
\]
Choosing $s_2$ large enough finishes the proof.
\end{proof}
\end{lemma}

Now we prove the main lemma of this alternative and along the proof we determine the lenght of the cylinder $Q^i$ by fixing $\lambda$ and consequently $c_i$.

\begin{lemma}
Assume \eqref{l:alt2} holds. Then the choice of $\lambda$ can be made so that
$$
u(x,t)\leq \mu_i^+ - \frac{\omega_i}{2^{\lambda+1}}, \qquad \text{a.e.} \quad \text{in} \quad \tfrac{1}{2}Q\l(\tfrac{c_i}{2}R_i^p,R_i\r).
$$
\begin{proof}
Define
\[
R_n=\frac{R_i}{2}+\frac{R_i}{2^{n+1}},
\qquad Q_{n}=Q\l(\tfrac{c_i}{2}R_n^p,R_n\r), \qquad B_n=B(R_n)
\]
and
\[
k_n = \mu_i^+ - \frac{\omega_i}{2^{\lambda+1}}\left(1 + \frac1{2^{n}}  \right)
\]
for $n=0,1,\dots$. Consider the energy estimate \eqref{l:energyest} written over the cylinders $Q_n$ and for the test functions $(u-k_n)_+ \varphi_n$, where $\varphi_n \in C_0^\infty(Q_n)$ vanishing on the parabolic boundary of $Q_n$ and such that $0\le\varphi_n\le 1$,
$\varphi_n=1$ in $Q_{n+1}$,
\begin{align}\label{gradient_estimates_3}
|\nabla \varphi_n|\le\frac{C2^{n}}{ R_i}
\quad\text{and}\quad \left(\frac{\partial \varphi_n}{\partial t}\right)_+\le \frac{C2^{pn}}{c_i R_i^p}.
\end{align}
Observe that when $(u-k_n)_+ \neq 0$,
$$\frac{\omega_i}{2}\leq \mu_i^+ - \frac{\omega_i}{2}\leq k_n \leq u\leq \mu_i^+\leq 2\omega_i $$
since $\mu_i^+=\mu_i^-+\omega_i$. Here the upper bound follows from assumption~\eqref{eq:Case I test}. Therefore the energy estimate reads as
\begin{align}\label{eq:energy_l3}
&\omega_i^{q-1}\esssup_{-\frac{c_i}{2}R_n^p<t<0}\dashint\limits_{B_n} \! (u{-}k_n)_+^{2}\varphi_n^{p}\,d\mu \,
+ \frac{c_i}{2}R_n^p\dashint\limits_{Q_n}\abs{\nabla(u{-}k_n)_+\varphi}^{p}\,d\nu \nonumber \\
&\leq C\,\frac{c_i}{2}R_n^p\brc{\dashint\limits_{Q_n}(u{-}k_n)_+^{p}\abs{\nabla \varphi_n}^{p}\,d\nu + C\omega_i^{q-1}\dashint\limits_{Q_n}(u{-}k_n)_+^2\varphi_n^{p-1}
\left(\frac{\partial \varphi_n}{\partial t}\right)_+\,d\nu} .
\end{align}
Using \eqref{gradient_estimates_3} and the estimate $(u-k_n)_+\leq2^{-\lambda}\omega_i$ we arrive at
$$
\esssup_{-\frac{c_i}{2}R_n^p<t<0} \,\dashint\limits_{B_n} \! (u{-}k_n)_+^{2}\varphi_n^{p}\,d\mu\leq C\,2^{np}\l(\frac{\omega_i}{2^\lambda}\r)^2 \frac{\nu\l(A_n\r)}{\nu\l(Q_n\r)}
$$
and
$$
R_n^p\dashint\limits_{Q_n}\abs{\nabla(u{-}k_n)_+\varphi}^{p}\,d\nu\leq  C\,2^{np}\l(\frac{\omega_i}{2^\lambda}\r)^p\frac{\nu\l(A_n\r)}{\nu\l(Q_n\r)},
$$
where $A_n=\left\{(x,t)\in Q_{n}:u(x,t)> k_n\right\}$.
Consider~\eqref{HoldSobo temp} and observe that it still holds for truncated functions $(u-k_{n})_+$ and for this sort of geometric setting. In a similar way to that in the proof of Lemma \ref{main_lemma1}, the left hand side can be estimated from below as
$$
\dashint_{Q_{n+1}}(u-k_{n})_+^{2(1-p/\kappa)+ p} \, d\nu  \geq  \l(\frac{\omega_i}{2^{\lambda+n+2}}\r)^{2(1-p/\kappa)+ p}\frac{\nu(A_{n+1})}{\nu(Q_{n+1})}.
$$
Substituting these last three estimates in~\eqref{HoldSobo temp}, written in terms of $(u-k_{n})_+$, we obtain
$$
\frac{\nu(A_{n+1})}{\nu(Q_{n+1})}\leq C^{n+1}\l(\frac{\nu\l(A_n\r)}{\nu\l(Q_n\r)}\r)^{2-p/\kappa}.
$$
By setting, as usual,
\[
Y_n = \frac{\nu(A_n)}{\nu(Q_n)}
\]
we find
\[
Y_{n+1} \le C^{n+1} Y_n^{2-p/\kappa}.
\]
Defining
$$\alpha_2 = C^{-1/(1-p/\kappa)+1-(1-p/\kappa)^2}$$
we conclude, using Lemma~\ref{geometric_convergence}, that if $Y_0\leq\alpha_2$ holds then $Y_n \to 0$, as $n\to \infty$. Now we set
\begin{equation}\label{eq:lambda chosen}
\lambda = \max\{1+\frac{1}{p-2},s_2\}.
\end{equation}
The choice trivially satisfies~\eqref{asp_lambda_1} (and $\lambda \geq s_2$)
and by the previous lemma we conclude that $\lambda$ can be chosen large enough so that $\nu(A_{0})/\nu(Q_{0})$ is as small as we please. Consequently, we reach the final result from the fact that $A_n\rightarrow0$ as $n\rightarrow0$.
\end{proof}
\end{lemma}

We obtain the reduction of the oscillation of the solution for the second alternative as a corollary of the previous lemma.

\begin{corollary}\label{cor:Case II conclusion 1}
Assume that~\eqref{eq:u between mu(i)pm}
holds in $Q^i=Q(c_iR_i^p,R_i)$ and that~\eqref{eq:Case I test} is satisfied. Suppose further that \eqref{l:alt2} is verified. Then
there exists a constant $\sigma_{II} \in (3/4,1)$, depending only on the data, such that
\begin{equation}\label{shrink}
    \essosc_{\frac12 Q\l(\frac{c_i}{2}R_i^p, R_i\r)}{u} \leq \sigma_{II} \omega_{i}.
\end{equation}
\begin{proof}
By the result of the previous lemma we have
$$
\esssup_{\frac12 Q\l(\frac{c_i}{2}R_i^p, R_i\r)} u
\leq \mu_i^+-\frac{\omega_i}{2^{\lambda+1}}.
$$
Set
\[
\mu_{i+1}^+=\mu_i^+ -\frac{\omega_i}{2^{\lambda+1}}\qquad \textrm{and} \qquad \mu_{i+1}^-=\mu_i^-.
\]
Then we have
$$
\essosc_{\frac12 Q\l(\frac{c_i}{2}R_i^p, R_i\r)}{u} \leq \mu_{i+1}^+ - \mu_{i+1}^- = \l(1-\frac{1}{2^{\lambda+1}}\r) \omega_{i}
$$
and the proof is complete with $\sigma_{II}=1-\frac{1}{2^{\lambda+1}}\in (3/4,1)$.
\end{proof}
\end{corollary}

\subsection{The case when~\eqref{eq:Case I test} fails}\label{subsec3}

If assumption~\eqref{eq:Case I test} does not hold, which can happen only for an index $i_0\geq 1$ due to assumptions~\eqref{bounds_inf} and~\eqref{first_bounds}, then we have
\begin{equation}\label{cor:Case II test}
\quad 4\mu_{i_0}^- > \omega_{i_0}  \ \Longleftrightarrow \ \mu_{i_0}^+ < 5\mu_{i_0}^-.
\end{equation}
This elliptic Harnack estimate implies that if this condition holds for an index $i_0$ then it continues to hold for all indices $i\geq i_0$.
By means of \eqref{cor:Case II test} we obtain $(4\mu_{i_0}^-)^{q-1}\leq \omega_{i_0}^{q-1}$ and this implies
$$
Q\l((4\mu_{i_0}^-)^{q-1}\l(\frac{\omega_{i_0}}{2^\lambda}\r)^{2-p}R_{i_0}^p,R_{i_0}\r)\subset Q^{i_0}.
$$
By setting
$$
{\bar R_{i_0}}=\frac {R_{i_0}}{4^{(1-q)/p}}\leq R_{i_0}
$$
we have
$$
{\bar Q^{i_0}}:=Q\l((\mu_{i_0}^-)^{q-1}\l(\frac{\omega_{i_0}}{2^\lambda}\r)^{2-p}{\bar R_{i_0}^p},{\bar R_{i_0}}\r)\subset Q^{i_0}.
$$
In this case we work with the following scalings factors
$$
\bar c_{i_0}=(\mu_{i_0}^-)^{q-1}\l(\frac{\omega_{i_0}}{2^\lambda}\r)^{2-p} \qquad \text{and} \qquad \bar d_{i_0}=(\mu_{i_0}^-)^{q-1}\l(\frac{\omega_{i_0}}{2}\r)^{2-p}
$$
and proceed in the same way as was done before - the only significant change appears in the proof of Lemma~\ref{main_lemma1}. Now, instead of \eqref{minus_upper_estimate}, we are allowed to use the more favorable estimate
\begin{equation*}
\begin{split}
\mathcal{J}((u{-}k)_-) &=q\int_0^{(u{-}k)_-}(k-\xi)^{q-1}\xi \, d\xi\\
&\le q u^{q-1}\int_0^{(u{-}k)_-}\xi\, d\xi \\
&\le q u^{q-1}\frac{(u{-}k)_-^2}{2}.
\end{split}
\end{equation*}
Observe that this estimate can be used in this case because, thanks to the assumption \eqref{cor:Case II test}, we have $u\geq\mu_i^->\omega_i/4$, for $i\geq i_0$, which implies that solution $u$ is strictly away from zero.

\subsection{Proving Theorem~\ref{main_theorem}}\label{subsec4}

We finally prove the H\"older continuity of $u$ presented in Theorem~\ref{main_theorem}  which is an immediate consequence of the following theorem.

\begin{theorem}\label{impliesholder}
Suppose that $u$ is a nonnegative weak solution of  equation~\eqref{equation} in $Q_{x,t}(R^2,2R)$. Then there are
positive constants $C$ and $\alpha$, both depending only on the data, such that
\[
\essosc_{Q_{x,t}(\varrho^p,\varrho)}{u} \leq C \left(\frac{\varrho}{R} \right)^{\alpha} \omega,
\]
for all $0<\varrho<R$.
\begin{proof}
After translation, we may assume that $(x,t) \equiv (0,0)$.
Take $\sigma := \max\{\sigma_I,\sigma_{II}\}$ and $\delta := \sigma^{(p-q-1)/p}2^{-\l(3+(\lambda-1)(p-2)/p\r)}$. Observe that this choice of $\delta$ gives
$$
c_{i+1}\l(\delta R_i\r)^p\leq d_i\l(\frac{R_i}{8}\r)^p \leq \frac{c_i}{2}\l(\frac{R_i}{2}\r)^p.
$$
Then, by Corollaries~\ref{cor:Case I conclusion} and~\ref{cor:Case II conclusion 1} we have
\[
\essosc_{Q(c_{i+1}(\delta^{i+1} R)^p ,\delta^{i+1} R)}{u} \leq \omega_{i+1} := \sigma^{i+1} \omega_0=\sigma^{i+1} \omega, \qquad i=0,1,\ldots.
\]
Since $c_{i} \geq 1$ for all $i=0,1,\dots$,  we finally obtain
\[
\essosc_{Q((\delta^{i} R)^p , \delta^i R)}{u} \leq  \sigma^{i} \omega, \qquad i=0,1,2,\ldots.
\]
Let $0<\varrho<R$. Then there exists $i \in \mathbb{N}$ such that $\delta^{i+1} R \leq \varrho \leq \delta^{i} R$. Take $$\alpha = \dfrac{\ln \sigma}{\ln \delta} \ .$$
Therefore, we get
\begin{eqnarray*}
\essosc_{Q((\delta^{i} R)^p , \delta^i R)}{u} & \leq &  \sigma^{i} \omega = \dfrac{1}{\sigma} \sigma^{i+1} \omega\\
& = & C \left(\delta^{i+1}\right)^\alpha \omega \ , \quad \mbox{for} \ C=1/\sigma\\
& \leq & C  \left(\frac{\varrho}{R} \right)^{\alpha} \omega.
 \end{eqnarray*}
This concludes the proof.
\end{proof}
\end{theorem}


\begin{thebibliography}{10}

\bibitem{AcerMing07}
Emilio Acerbi and Giuseppe Mingione.
\newblock Gradient estimates for a class of parabolic systems.
\newblock {\em Duke Math. J.}, 136(2):285--320, 2007.

\bibitem{BogeDuzaMing11}
Verena B\"ogelein, Frank Duzaar, and Giuseppe Mingione.
\newblock Degenerate problems with irregular obstacles.
\newblock {\em Journal f\"ur die reine und angewandte Mathematik (Crelles
  Journal)}, (650):107--160, 2011.

\bibitem{Buck99}
Stephen~M. Buckley.
\newblock Is the maximal function of a {L}ipschitz function continuous?
\newblock {\em Ann. Acad. Sci. Fenn. Math.}, 24(2):519--528, 1999.



\bibitem{DiBe93}
Emmanuele DiBenedetto.
\newblock {\em Degenerate parabolic equations}.
\newblock Universitext. Springer-Verlag, New York, 1993.


\bibitem{DiBeGianVesp08}
Emmanuele DiBenedetto, Ugo Gianazza, and Vincenzo Vespri.
\newblock Harnack estimates for quasi-linear degenerate parabolic differential
  equations.
\newblock {\em Acta Math.}, 200(2):181--209, 2008.

\bibitem{DiBeGianVesp10a}
Emmanuele DiBenedetto, Ugo Gianazza, and Vincenzo Vespri.
\newblock Forward, backward and elliptic {H}arnack inequalities for
  non-negative solutions to certain singular parabolic partial differential
  equations.
\newblock {\em Ann. Sc. Norm. Super. Pisa Cl. Sci. (5)}, 9(2):385--422, 2010.

\bibitem{DiBeUrbaVesp04}
Emmanuele DiBenedetto, Jos{\'e}~Miguel Urbano, and Vincenzo Vespri.
\newblock Current issues on singular and degenerate evolution equations.
\newblock In {\em Evolutionary equations. Vol. I}, Handb. Differ. Equ., pages
  169--286. North-Holland, Amsterdam, 2004.

\bibitem{DiazThelin94}
Jesus Ildefonso Diaz and Francois De Thelin
\newblock On a Nonlinear Parabolic Problem Arising in Some Models Related to Turbulent Flows.
\newblock {\em SIAM J. Math. Anal.}, 25(4):1085–1111, 1994.


\bibitem{FornGian10}
Simona~Fornaro and Ugo~Gianazza.
\newblock Local properties of non-negative solutions to some doubly non-linear
  degenerate parabolic equations.
\newblock {\em Discrete Contin. Dyn. Syst.}, 26(2):481--492, 2010.

\bibitem{FornSosi08}
Simona~Fornaro and Maria~Sosio.
\newblock Intrinsic {H}arnack estimates for some doubly nonlinear degenerate
  parabolic equations.
\newblock {\em Adv. Differential Equations}, 13(1-2):139--168, 2008.



\bibitem{Hajl03}
Piotr~Hajlasz.
{\em Sobolev spaces on metric-measure spaces}.
\newblock Heat kernels
and analysis on manifolds, graphs, and metric spaces, (Paris, 2002)
Contemp. Math., Amer. Math. Soc., 338, Providence, RI 173--218, 2003.

\bibitem{HajlKosk00}
Piotr Hajlasz and Pekka Koskela.
\newblock Sobolev met {P}oincar\'e.
\newblock {\em Mem. Amer. Math. Soc.}, 145(688):x+101, 2000.

\bibitem{HeinKilpMart93}
Juha Heinonen, Tero Kilpel{\"a}inen, and Olli Martio.
\newblock {\em Nonlinear potential theory of degenerate elliptic equations}.
\newblock Oxford Mathematical Monographs. The Clarendon Press Oxford University
  Press, New York, 1993.

\bibitem{HenrUrba05}
Eurica Henriques and Jos{\'e}~Miguel Urbano.
\newblock On the doubly singular equation {$\gamma(u)_t=\Delta_pu$}.
\newblock {\em Comm. Partial Differential Equations}, 30(4-6):919--955, 2005.

\bibitem{Ivan94}
Alexander~V. Ivanov.
\newblock H\"older estimates for equations of fast diffusion type.
\newblock {\em Algebra i Analiz}, 6(4):101--142, 1994.

\bibitem{KeitZhon08}
Stephen Keith and Xiao Zhong.
\newblock The {P}oincar\'e inequality is an open ended condition.
\newblock {\em Ann. of Math. (2)}, 167(2):575--599, 2008.

\bibitem{KinnKuus07}
Juha Kinnunen and Tuomo Kuusi.
\newblock Local behaviour of solutions to doubly nonlinear parabolic equations.
\newblock {\em Math. Ann.}, 337(3):705--728, 2007.

\bibitem{KinnShan01}
Juha Kinnunen and Nageswari Shanmugalingam.
\newblock Regularity of quasi-minimizers on metric spaces.
\newblock {\em Manuscripta Math.}, 105(3):401--423, 2001.

\bibitem{KLSU}
Tuomo Kuusi, Rojbin Laleoglu, Juhana Siljander and Jos\'e Miguel Urbano.
\newblock H\"older continuity for Trudinger's equation in measure spaces.
\newblock {\em Calc. Var. Partial Differential Equations}, 45(1-2):193--229, 2012.

\bibitem{Kuus08}
Tuomo Kuusi.
\newblock Harnack estimates for weak supersolutions to nonlinear degenerate
  parabolic equations.
\newblock {\em Ann. Sc. Norm. Super. Pisa Cl. Sci. (5)}, 7(4):673--716, 2008.

\bibitem{KuusSiljUrba10}
Tuomo Kuusi, Juhana Siljander, and Jos\'e~Miguel Urbano.
\newblock H\"older continuity to a doubly nonlinear parabolic equation.
\newblock {\it Indiana Univ. Math. J.}, to appear.

\bibitem{PorzVesp93}
Maria~M. Porzio and Vincenzo Vespri.
\newblock H\"older estimates for local solutions of some doubly nonlinear
  degenerate parabolic equations.
\newblock {\em J. Differential Equations}, 103(1):146--178, 1993.

\bibitem{SaCo92}
Laurent Saloff-Coste.
\newblock A note on {P}oincar\'e, {S}obolev, and {H}arnack inequalities.
\newblock {\em Internat. Math. Res. Notices}, (2):27--38, 1992.

\bibitem{SaCo02}
Laurent Saloff-Coste.
\newblock {\em Aspects of {S}obolev-type inequalities}, volume 289 of {\em
  London Mathematical Society Lecture Note Series}.
\newblock Cambridge University Press, Cambridge, 2002.

\bibitem{Trud68}
Neil~S. Trudinger.
\newblock Pointwise estimates and quasilinear parabolic equations.
\newblock {\em Comm. Pure Appl. Math.}, 21:205--226, 1968.

\bibitem{Urba08}
Jos{\'e}~Miguel Urbano.
\newblock {\em The Method of Intrinsic Scaling. A systematic approach to
  regularity for degenerate and singular PDEs}, volume 1930 of {\em Lecture
  Notes in Mathematics}.
\newblock Springer-Verlag, Berlin, 2008.

\bibitem{Vesp92}
Vincenzo Vespri.
\newblock On the local behaviour of solutions of a certain class of doubly
  nonlinear parabolic equations.
\newblock {\em Manuscripta Math.}, 75(1):65--80, 1992.



\end{thebibliography}
\bibliographystyle{plain}

\def\cprime{$'$} \def\cprime{$'$}

\bigskip
\noindent Addresses:

\noindent E.H. \& R.L. : Departamento de Matem\'{a}tica, Universidade de Tr\'as-os- Montes e Alto Douro,
Ap. 1013, 5001-801 Vila Real, Portugal. \\
\noindent
E-mail (E.H.): {\tt eurica@utad.pt}

\noindent
E-mail (R.L.): {\tt rojbin@utad.pt}\\

\end{document}